\documentclass{article}
\usepackage{amssymb}
\usepackage{amsmath}
\usepackage{tikz}
\usepackage[utf8]{inputenc}
\usepackage[russian]{babel}
\usepackage{amssymb,amsfonts,amsmath,mathtext}
\usepackage{mathrsfs}
\usepackage{tikz}

\usepackage[section,above,below]{placeins}
\usepackage{flafter}
\usepackage{todonotes}
\usepackage{floatrow}

\def\CC{{\mathbb C}}

\def\sm{{\setminus}}            
\def\bc{{\rm bc}}                    

\usepackage{theorem}

\theoremstyle{plain}
\newtheorem{Th}{Theorem}[section]
\newtheorem{Prop}{Proposition}[section]
\newtheorem{Cor}{Corollary}[section]
\newtheorem{Ax}{Axiom}
\newtheorem{Le}{Lemma}[section]
\newtheorem{Rem}{Remark}[section]

\theoremstyle{definition}

\newenvironment{proof}{\par\noindent{\bf Proof.}}
{\hfill$\scriptstyle\blacksquare$}

\newcommand{\btu}{\bigtriangleup}

\usepackage[normalem]{ulem}  

\def\CC{{{\mathbb{C}}}}

\def\CC{{\mathbb C}}

\date{September 7, 2021}

\title{Transition polynomial as a weight system for binary delta-matroids \thanks{The article was prepared within the framework of the Academic Fund Program at the National Research University Higher School
of Economics (HSE) in 2020–2021 (grant 20-04-010) and supported within the framework of a subsidy granted
to the HSE by the Government of the Russian Federation for the implementation of the Global Competitiveness Program.}}
\author{Alexander Dunaykin \thanks{International Laboratory of Cluster Geometry National Research University Higher School of Economics, email: {\tt
      alexdunaykin@gmail.com}}, Vyacheslav Zhukov \thanks{International Laboratory of Cluster Geometry National Research University Higher School of Economics, email: {\tt slava.zhukov@list.ru}}}

\begin{document}
\def\figurename{Fig.}

\maketitle
\renewcommand\abstractname{Abstract}
\renewcommand\refname{References}

\begin{abstract}

  To a singular knot $K$ with $n$ double points, one can associate a chord diagram with $n$ chords.
  A chord diagram can also be understood as a $4$-regular graph endowed with an oriented Euler circuit.
  L.~Traldi introduced a polynomial invariant for such graphs,
  called a transition polynomial.
  We specialize this polynomial
  to a multiplicative weight system, that is, a function on chord diagrams satisfying
  $4$-term relations and determining thus a finite type knot invariant.
  We prove a similar statement for the transition polynomial
  of general ribbon graphs and binary delta-matroids defined by
  R.~Brijder and H.~J.~Hoogeboom, which defines, as a consequence,
  a finite type invariant of links.

\end{abstract}

{\bf Перевод аннотации}

Особому узлу $K$ с $n$ двойными точками сопоставляется хордовая диаграмма с $n$ хордами. Хордовую диаграмму можно также понимать как 4-регулярный граф с выделенным ориентированным эйлеровым циклом. Л.~Тральди ввел инвариант
таких графов, называемый многочленом переходов.
Выбирая специальные параметры, мы превращаем этот многочлен в
весовую систему, то есть функцию на хордовых диаграммах, которая удовлетворяет четырёхчленному соотношению, а значит определяет инвариант узлов конечного типа. Аналогичное утверждение мы доказываем и для
многочлена переходов общих вложенных графов и бинарных дельта-матроидов,
введенного Р.~Брийдером и Х.~Хугебумом, определяя, тем самым, инвариант
зацеплений конечного типа.

{\bf Short title}

Transition polynomial as a weight system

{\bf List
  of key words or phrases}

knot, link, finite type invariant of knots, chord diagram, transition polynomial, delta-matroid.


\section{Introduction}

A \emph{chord diagram of order~$n$} is an oriented circle
with~$2n$ distinct points on it, split into~$n$ disjoint pairs and considered up to
orientation preserving diffeomorphisms of the circle. A function on chord diagrams is
a \emph{weight system} provided it satisfies Vassiliev's $4$-term relations (see precise definitions
in the next section). Vassiliev has shown that functions on chord diagrams with~$n$ chords
obtained from knot invariants of order at most~$n$ satisfy the $4$-term relations,
and Kontsevich proved that these are essentially the only restrictions, that is,
a knot invariant is associated to each weight system.

We start with the following construction.

First, by contracting each chord of a chord diagram to a vertex, we make the diagram
into a $4$-regular graph, that is, a graph in which all the vertices are $4$-valent.
The set of vertices of this graph is
in one-to-one correspondence with the set of chords in the initial chord diagram.
The graph is also endowed with a distinguished oriented Euler circuit, which is the supporting
circle of the chord diagram.

To such a pair (namely, a $4$-regular graph with an oriented Euler circuit), L.~Traldi associates
the weighted transition polynomial. This polynomial depends on three parameters,
denoted $s,t$, and~$u$. Our first main result consists in showing that for $u=-t$ the
weighted transition polynomial is a weight system (taking values in the polynomial
ring $\CC[s,t,x]$, the variable~$x$ being the argument of the transition polynomial).

A chord diagram also can be interpreted as an embedded graph
with a single vertex. More generally, to a singular link one associates
an embedded graph with several vertices, whose number equals the
number of connected components of the link.
Our second main result is that the transition polynomial for delta-matroids
defined in~\cite{BRHO} satisfies, after the same specialization,
the $4$-term relations for binary
delta-matroids introduced in~\cite{lando2016delta} and defines
thus a finite type invariant of links.

The paper is organized as follows.

In Section~\ref{s2} we introduce the required definitions and formulate the main result for chord diagrams.
Section~\ref{s3} is devoted to its proof. Section~\ref{s4}
is devoted to the construction of an extension
of the transition polynomial to arbitrary embedded graphs. We finish with
constructing an extension of our invariant to binary delta-matroids.

The authors are grateful to Sergei Lando for the advice to study the theory of Vassiliev knot invariants and for pointing out
the construction of Lorenzo Traldi in connection with them. This article wouldn't appear without his advice. The authors also are grateful to unknown referee for useful comments
allowing them to seriously improve presentation.

\section{Definitions and statement of the main result}\label{s2}

\subsection{Chord diagrams and weight systems}

A chord diagram of order~$n$ is an oriented circle
with~$2n$ distinct points on it, split into~$n$ disjoint pairs and considered up to
orientation preserving diffeomorphisms of the circle.
A function~$f$ on chord diagrams with values in a commutative ring
is called a {\it weight system\/} if it satisfies the {\it $4$-term relations\/}
shown in Fig.~\ref{4term}. Here we pick a chord diagram~$C$ and two chords
with neighboring ends in it, and construct the other three diagrams as shown.
All the four circles are assumed to be oriented counterclockwise.
The four diagrams in the picture may contain other chords with the ends
on the dotted arcs, which are the same for all four of them. An equivalent way to look at this is to consider functions
on the vector space $M$ spanned by all chord diagrams over the field $\CC$ factored over all $4$-term relations. The vector space
$M$ has a ring structure. In order to multiply two chord diagrams, $C_1$ and $C_2$, we cut the supporting circle of each
diagram at an  arbitrary point different from the endpoints of the chords and glue the resulting arcs together to
form a new supporting circle in an orientation-preserving way as it is done in Fig.~\ref{chordmult}.
Modulo $4$-term relations, the result does not depend on the way we have chosen the cutting points.
For the basics
of Vassiliev knot invariants we refer the reader to
Chapter~6 in the book by S. Lando and A. Zvonkin \cite{embedded graphs}.

\begin{figure}[h]
  \center{\includegraphics[width=0.65\linewidth]{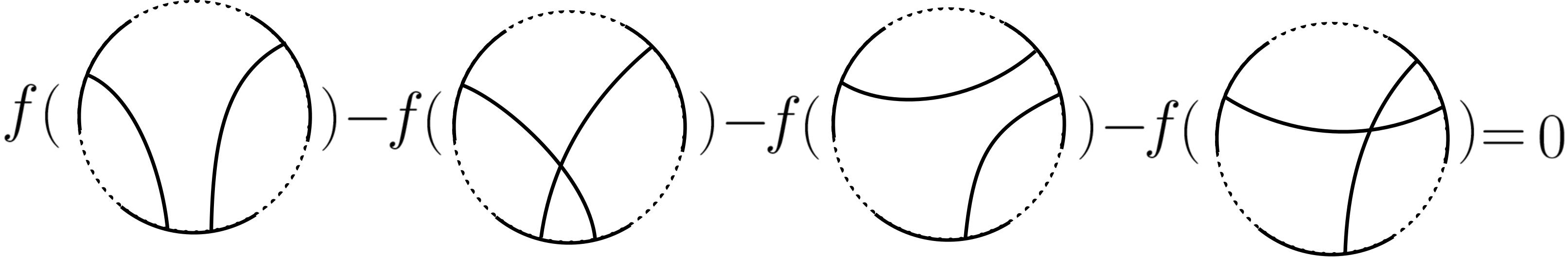}}
  \caption{4-term relation}
  \label{4term}
\end{figure}

\begin{figure}[h]
  \begin{picture}(400,60)(-70,0)
    \includegraphics[width=0.7\linewidth]{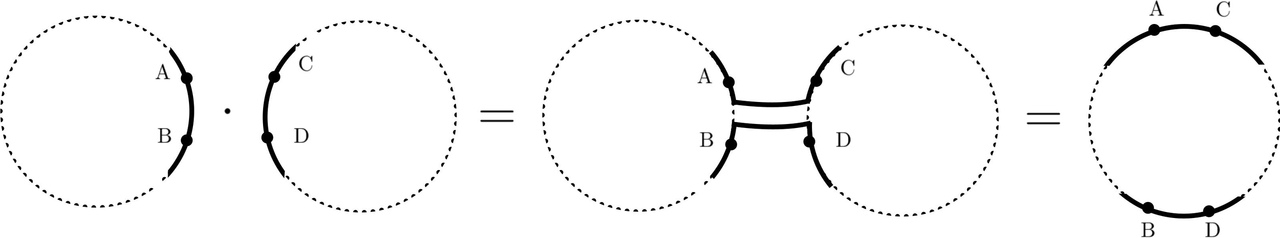}
  \end{picture}
  \caption{Multiplication of chord diagrams}
  \label{chordmult}
\end{figure}

\subsection{4-regular graphs, Greek labelings, and weighted tran\-si\-tion polynomials}

As discussed above, a 4-regular graph $G$ is a graph in which each vertex is 4-valent. By contracting the chords of a chord diagram $C$, we make it into
a $4$-regular graph $G=G(C)$ endowed with an oriented Euler circuit.
In this paper, we will look at an oriented Euler
circuit from two points of view. Firstly, we may interpret an Euler circuit as a sequence of half-edges $h_1$,
$h_2$, ... , $h_{4n}$ considered up to cyclic permutations of its entries, where $n$ is the number of vertices in $G$, such that

1. Each half-edge enters the sequence once and two half-edges
with consecutive indices either belong to the same
edge or are incident to the same vertex.

2. If $h_k$ and $h_{k+1}$ belong to the same edge, then $h_{k+1}$ and $h_{k+2}$ are incident to the same vertex.

3. If $h_k$ and $h_{k+1}$ are incident to the same vertex, then $h_{k+1}$ and $h_{k+2}$ belong to the same edge.

The second way to look at an oriented Euler circuit is to say that it is an immersion of the standard oriented
circle to $G$ such that each point of $G$ except the vertices has exactly one pre-image and each vertex has two
pre-images. This construction is considered up to homotopy in the class of such maps.

\subsubsection{Transitions and their Greek labeling}

Let $G$ be a 4-regular graph and let $K$ be an oriented Euler circuit in it. At each vertex $v$ of $G$, there are 4
half-edges incident to $v$. They form the set $H(v)$. There are three ways to split $H(v)$ into two disjoint
$2$-element subsets. These three partitions form the set $T(v)$, its elements are called the {\it transitions \/} at $v$.
The Euler circuit $K$ allows us to assign a type to any transition. We will mark the types with the Greek letters
$\phi$, $\chi$ and $\psi$.

Pick one of the two half-edges entering $v$ (we call this half-edge the {\it starting\/} one);
choosing a pair for this half-edge determines the transition completely.
There are three cases. If the pair to the starting half-edge is the one that follows it immediately along the Euler
circuit, then
we say that this transition belongs to type~$\phi$,
if it is the other leaving half-edge,
then this is a $\chi$-transition and if it is the other entering half-edge,
then this is a $\psi$-transition as  illustrated in Fig. \ref{symboling} (the letter `o' denotes
the starting half-edge).
Note that if we choose the other half-edge entering $v$ for the starting one,
then the types of the transitions will be the same.

\noindent
\begin{figure}
  \begin{picture}(400,60)(-100,0)
    \includegraphics[width=0.5\linewidth]{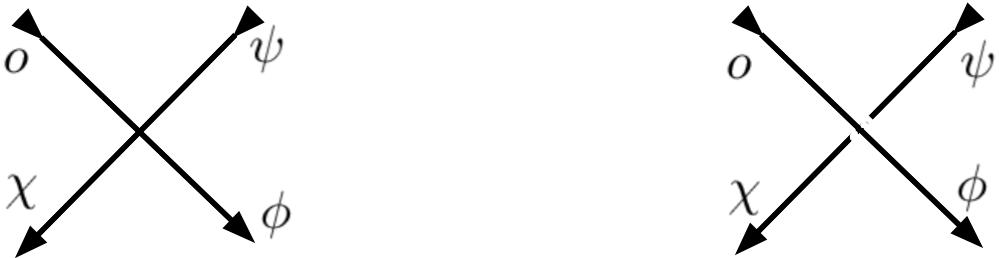}
  \end{picture}
  \caption{The Greek labeling of transitions (on the left) with respect to the specified Euler circuit (on
    the right)}
  \label{symboling}
\end{figure}

\subsubsection{Weighted transition polynomials}

A {\it circuit partition\/} $P$ of a 4-regular graph $G$ with~$n$ vertices is an $n$-tuple of transitions, one at
each vertex. Given a circuit partition $P$ of $G$,
we first erase all the vertices of $G$ and then we glue in pairs the free ends of
half-edges that were paired in some transition from $P$. Since we have taken one transition at each vertex of
$G$, each half-edge of $G$ participates exactly once in a transition from $P$ and we obtain a disjoint family of
circles. Let their number be $c(P)$.
Let ${\cal P}(G)$ be the set of all circuit partitions of $G$. The set of all transitions of $G$ is denoted by ${\cal T}(G)$. A
{\it weight function\/} is a map from ${\cal T}(G)$ to a commutative ring. For a given weight function~$w$,
the {\it weighted transition polynomial\/} $Q_w$ is the sum of the monomials that correspond to circuit partitions of $G$. The monomial for a given circuit
partition $P$ is $x^{c(P)-1}$ times the product of the weights of all transitions in $P$, so that
$$
Q_w(G)=\sum_{P \in P(G)} \prod_{v \in V(G)} w(T(v) \cap P) x^{c(P)-1}.
$$

The weighted transition polynomial was introduced by F. Jaeger \cite{transition polynomial}.

\subsection{Statement of the first main theorem}

If we define the weight function in such a way that it takes on a transition
values depending only on the type~$\phi,\psi$, or $\chi$ of the transition,
then we obtain Traldi's transition polynomial.
In this section, we introduce the function $Q$ taking chord diagrams to
elements of $\CC[s,t,x]$ as a specialization of Traldi's transition polynomial.
Its value on a chord diagram $C$ is the weighted transition polynomial
of the corres\-pon\-ding $4$-regular graph $G(C)$.
We attach to transitions in $G$ weights according to their types
with respect to $E(D)$. All the $\phi$-transitions are assigned the weight $s$,
all the $\psi$-transitions are assigned the weight $t$, and all the $\chi$-transitions are assigned the weight
$-t$.

\begin{Th}\label{tm1}
  The function $Q$ is a multiplicative weight system.
\end{Th}

\section{Proof of Theorem~\ref{tm1}}\label{s3}

Instead of counting the number of connected components in a circuit partition $P$ of a $4$-regular graph $G(C)$ with
an oriented Euler circuit $K=K(C)$,
we can count the number of connected components of the boundary of the
ribbon graph with one vertex corresponding to the chord
diagram~$C$ and the partition~$P$.
Let~$P$ be a circuit partition. Assign the corresponding Greek letters to the chords of the chord diagram~$C$;
such a marked chord diagram will be denoted by~$C(P)$. Associate to the marked chord diagram~$C(P)$ the ribbon
graph~$R(P)$
by attaching the disc to the supporting circle of~$C$ and
replacing every chord with marking~$\chi$ by a ribbon,
every chord with marking~$\psi$ by a half-twisted ribbon,
and erasing every chord with marking~$\phi$.

The value of the function $Q$ on a chord diagram with $n$ chords is a sum of $3^n$ monomials. Each monomial
corresponds to the choice of Greek letters at each of the chords.
Figure~\ref{onechord} shows how the value of~$\Psi$ on a chord diagram
is constructed, for a chosen chord and all its possible markings with the Greek
letters, assuming the markings on all the other chords are fixed.

\noindent
\begin{figure}
  \center{\includegraphics[width=0.65\linewidth]{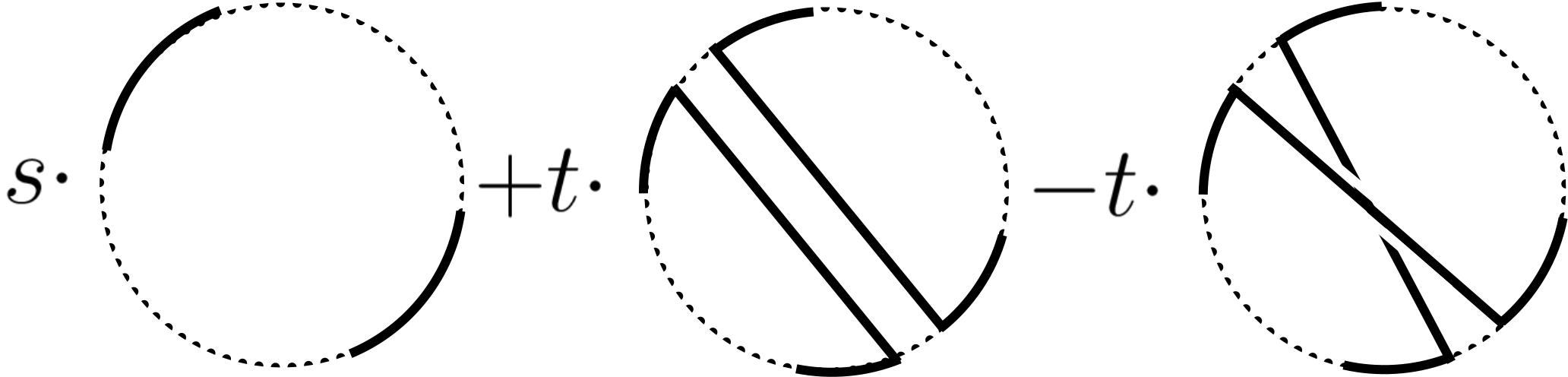}}
  \caption{Impact of the choices of the marks on a specific chord}
  \label{onechord}
\end{figure}

\begin{figure}
  \center{\includegraphics[width=\linewidth]{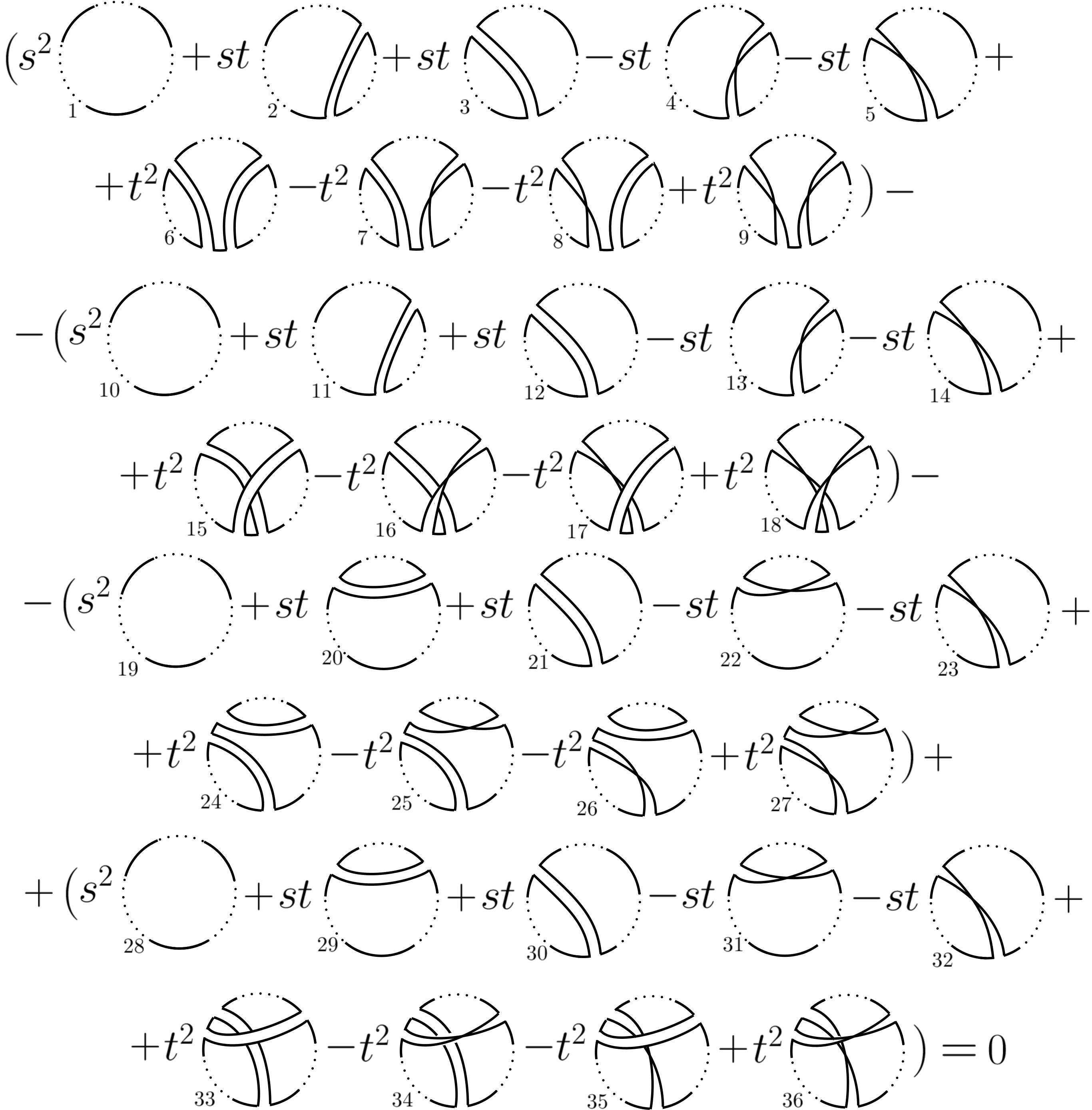}}
  \caption{Checking the $4$-term relation for $Q$}
  \label{4termpsi}
\end{figure}

In Fig.~\ref{4termpsi}, the $4$-term relation is checked. We take a chord diagram $C$ and pick
two chords $a$ and $b$ with neighboring ends in it.
We are going to show that the corresponding $4$-term relation is satisfied not just for
the whole function~$Q$, but for each subsum in it corresponding to a given
choice of Greek letters for all chords but~$a$ and~$b$.
Each bracket of the $4$-term expression contains 9 monomials in $Q$ and each monomial is
the product of the weights of $a$ and $b$ and $x^{c(P)-1}$ where $c(P)$ is the number of connected components of
the boundary of the ribbon graph associated to the
partition~$P$, times the product of the weights of all other chords, which are
the same for all the $4$ terms. The summands are numbered
(the number is shown in the brackets under the coefficient of the diagram with this number).
The paired summands below differ only by the sign:

\begin{gather*}
  (1,10); (2,11); (3,12); (4,13); (5,14); (6,24); (7,25);\\
  (8,36); (9,35); (15,33); (16,34); (17,27); (18,26);\\
  (19,28); (20,29); (21,30); (22,31); (23,32)
\end{gather*}

It is obvious that the two paired ribbon graphs in each pair are homeomorphic to one another, whence have the same number of boundary components.
Theorem~\ref{tm1} is proved.


\section{The extension of the $Q$-polynomial to ribbon graphs}\label{s4}

Above, we restricted our attention to
ribbon graphs with a single vertex in order to check the $4$-term relation for
the $Q$-function on chord diagrams. From now on we omit this restriction
and consider arbitrary ribbon graphs. We present a natural way to define the $Q$-polynomial
on a ribbon graph $R$ as an analogous specification of a transition polynomial for a medial graph of $R$. Similarly to the case of chord
diagrams, we attach a Greek letter to every ribbon (this data is denoted by $L$). Then we take the product of
weights of all Greek letters in $L$ ($s$ for $\phi$, $t$ for $\chi$ and $-t$ for $\psi$) and $x^{c(R(L))-1}$, where
$c(R(L))$ is the number of connected components of the boundary of the ribbon graph $R(L)$. The latter is constructed from
$R$ by half-twisting all the ribbons endowed with the letter $\psi$ and erasing all the ribbons endowed with the letter $\phi$.
The polynomial~$Q$ is then defined as the result of the summation over all states~$L$.

The {\it $4$-term relation for ribbon graphs \/} is shown in the upper row in Fig.~\ref{ribbon4}. Here we pick a ribbon graph $R$ and two ribbons with neighboring ends in it, and
construct the other three ribbon graphs as shown. The four ribbon graphs in the picture may contain
other ribbons, which are the same for all four of them.

\begin{figure}
  \center{\includegraphics[width=\linewidth]{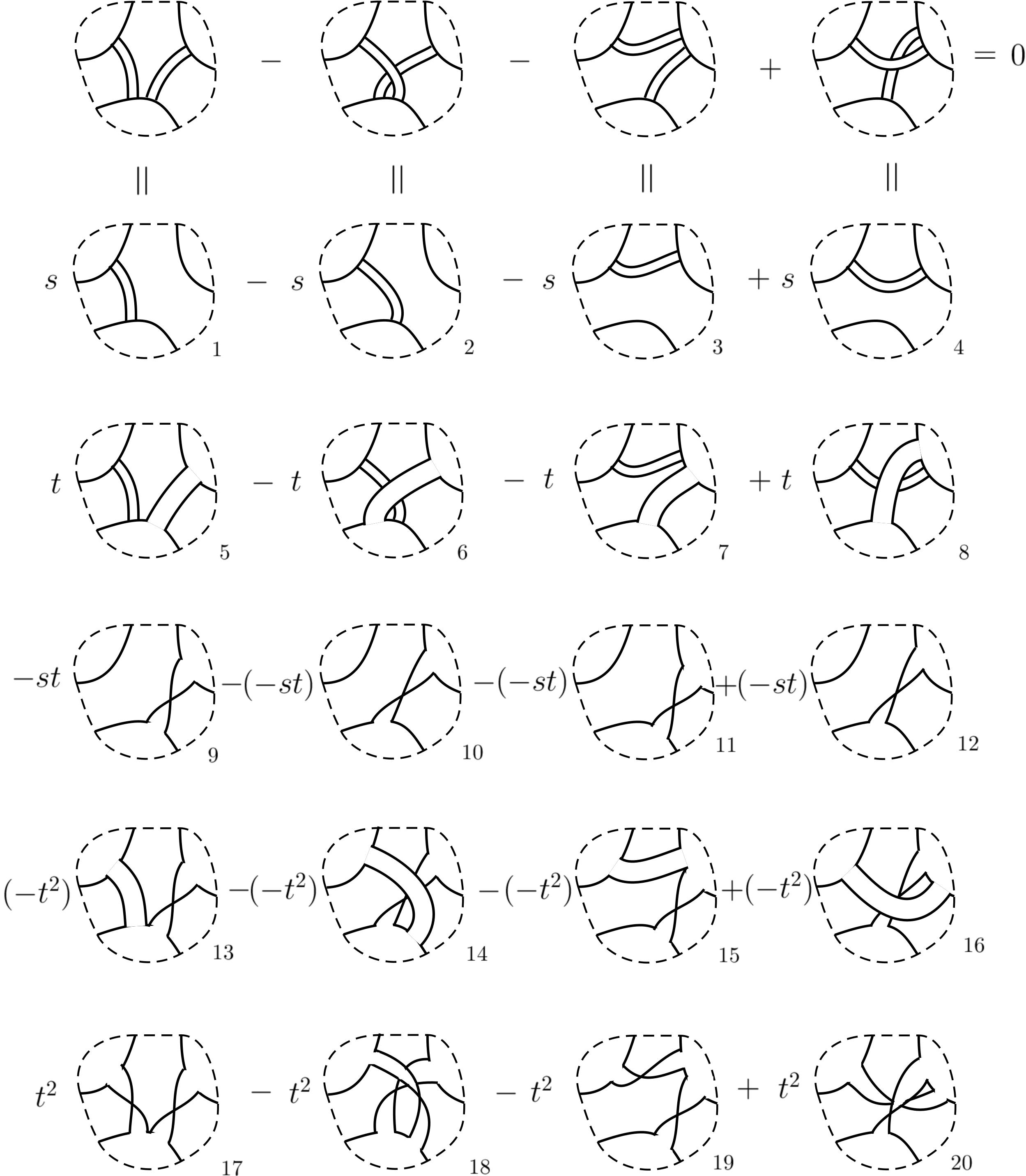}}
  \caption{Checking the 4-term relation for ribbon graphs}
  \label{ribbon4}
\end{figure}

The 4-term relation for our polynomial on ribbon graphs is checked in the same way as it was done for chord diagrams, see Fig. \ref{ribbon4}.

In more detail, each column in Fig.~\ref{ribbon4} represents an expression for the $Q$-polynomial
of a ribbon graph with two distinguished ribbons.
  Elements of the second and the third lines are obtained from the corresponding
  elements of the first line by deleting the second ribbon or contracting it, respectively.
  Elements of the fourth, fifth, and sixth lines are obtained from the corresponding elements of the
  first line by twisting and then contracting the second ribbon, while the first ribbon
is, respectively, deleted, contracted, or twisted and contracted.

The paired summands below differ only by sign:
  \begin{gather*}
    (1, 2); (3, 4); (5, 7); (6, 8); (9, 10); (11, 12); (13, 20);\\
    (14, 19); (15, 18); (16, 17); \\
  \end{gather*}
it is obvious that the two paired ribbon graphs in each pair are
  homeomorphic to one another, whence have the same number of boundary components.




\section{Weighted transition polynomial for binary
  delta-matroids}
\label{sec:weight-trans-polyn}

In this section we study the weighted transition polynomial for binary $\btu$-matroids.
The notions of a $\btu$-matroid and a binary $\btu$-matroid were introduced by Bouchet \cite{bouchet1989maps}.
Our presentation below follows that of \cite{lando2016delta}, where the $4$-term relations and
the Hopf algebra of binary delta-matroids were introduced.

Weighted transition polynomial~$Q$ for delta-matroids (and, more generally, for
arbitrary set systems) was defined in~\cite{BRHO}.
To each ribbon graph $R$, a binary delta-matroid $D(R)$ is associated.
In what follows, we always assume, without special indication, that the ribbon graphs in question are connected.
(Note that it does not make sense to consider disconnected ribbon graphs since they do not have quasitrees,
whence the corresponding delta-matroid is improper.)
The polynomial~$Q$
for binary delta-matroids possesses the property that
$Q(D(R))=Q(R)$.
We also prove that the 4-term relations for binary delta-matroids introduced in \cite{lando2016delta} are satisfied
for the transition polynomial.
We start with basic notions from the theory of delta-matroids.
Then, following~\cite{BRHO}, we define the transition polynomial for binary delta-matroids and its specification~$Q$, and
prove that it satisfies the $4$-term relations.

\subsection{Basics of delta-matroids}
\label{sec:basics-delta-matr}

A \emph{set system} is a pair $(E, S)$ where $E$ is a finite set and $S \subset 2^E$. The set $E$ is called the
\emph{ground set}
and elements of $S$ are called \emph{feasible sets}. Two set systems $(E_1, S_1)$ and $(E_2, S_2)$ are said to be isomorphic
if there exists a bijection $f: E_1 \rightarrow E_2$ such that $f(S_1)=S_2$. Below, we do not distinguish between isomorphic
set systems.



A \emph{delta-matroid} is a set system $(E, S)$, with a non-empty $S$, satisfying the following Symmetric Exchange Axiom:

\begin{Ax}[SEA]
  For any two feasible sets $X$ and $Y$ and any $a \in X \btu Y$ there is $b \in X \btu Y$ {\rm(}which is allowed to be equal to $a${\rm)} such that $X \btu \{a, b\}$ is feasible {\rm(}in the case $b=a$,  $X\Delta\{a\}\in S${\rm)}.
\end{Ax}

Here and below $\btu$ denotes the symmetric difference operation on
pairs of sets.

\label{sec:basics-delta-matr-1}
To any ribbon graph $R$, we assign a delta-matroid $D(R)=(E(R),S(R))$.
Here $E(R)$ is the set of edges of $R$ and the feasible sets are those subsets of $E$ that induce a ribbon subgraph whose boundary
consists of a single connected component ({\it quasitrees}).

To a simple graph $G$, the delta-matroid $D(G)=(E(G), S(G))$ is associated. The ground set $E(G)$
is the set of vertices of $G$, $E=V(G)$. A subset $A\subset E(G)$ is
feasible if the adjacency matrix of the subgraph of~$G$ induced by $A$ is nondegenerate over ${\mathbb F}_2$ (and empty set is feasible by definition). A delta-matroid $D$ is
said to be \emph{graphic\/} if there exists a graph $G$ such that $D=D(G)$.

Let $(E,  S)$ be a $\btu$-matroid and let $A\subset E$; then
the \emph{partial duality} $(E, S)* A$ of $(E,  S)$ \emph{by the set} $A$ is defined by
$(E, S)* A = (E, \{F\subset E | F\btu A\in  S\})$.
{\rm(}If $A$ is a one-element set, $A=\{a\}$,  we simply write
$(E, S)* a$ instead of $(E,  S)*\{a\}${\rm)}.

A delta-matroid $D$ is said to be \emph{binary\/} if there exists a graphic delta-matroid $D{'}=(E, S)$ and a set $A\subset E$
such that $D=D{'}* A=(E, S)* A$.

\begin{Rem}
  For a ribbon graph $R$, the delta-matroid $D(R)$
  is a {\em binary delta-matroid}.
\end{Rem}

The following statement shows, in particular, that the delta-matroid of a ribbon graph with
a single vertex coincides with the delta-matroid of the intersec\-tion graph of the corresponding chord diagram.

\begin{Th}
  \label{sec:basics-delta-matr-2}
  Let $C$ be a chord diagram and let $\Gamma(C)$ be its intersection graph,
  $A(\Gamma(C))$ being its adjacency matrix over ${\mathbb F}_2$, then
  $corank(A(\Gamma(C))) = \bc(C)-1$ where $\bc(C)$ is the number of boundary components 
  of $C$.
\end{Th}

Recall that the intersection graph $\Gamma(C)$ of a chord diagram~$C$
is the graph whose vertices are in one-to-one correspondence with
the chords of~$C$, two vertices being connected by an edge iff the ends
of the corresponding chords alternate along the circle. A proof of this theorem can be found in \cite{mellor2000intersection}, \cite{moran1984chords}, \cite{soboleva2001vassiliev}.

An element $a$ of a $\btu$-matroid $(E,S)$ is a \emph{coloop\/} if for each $F\in  S$ we have $F\ni a$,
and it is a \emph{loop\/} if for any $F\in  S$ we have $F\not\ni a$.
These definitions mimic ones for ribbon graphs, where a coloop is usually known as a {\it bridge}.

Let $(E,  S)$ be a $\btu$-matroid, and $a\in E$, then
$(E,  S)\sm a$ {\sl is the result of deleting $a$}:
\[
  (E,  S)\sm a=
  \begin{cases}
    (E\setminus\{a\}, \{F\subset E\setminus \{a\} | F\in  S\})&
    \text{if }a\text{ is not a coloop}\\
    (E\setminus\{a\}, \{F\subset E\setminus \{a\} | F\cup\{a\}\in  S\})&\text{otherwise}\\
  \end{cases}
\]

\noindent We denote by
$(E,  S) / a$ the result of contracting 
$a$:

\[
  (E,  S) / a=
  \begin{cases}
    (E\setminus\{a\}, \{F\subset E\setminus \{a\} | F\cup\{a\}\in  S\})&
    \text{if }a\text{ is not a loop}\\
    (E\setminus\{a\}, \{F\subset E\setminus \{a\} | F\in  S\})&\text{otherwise}\\
  \end{cases}
\]



For a delta-matroid \(D=(E,S)\), define the function~$d_D$ on the subsets of its ground set by the formula
\(d_D(A)=\min_{F\in S}|A\Delta F|\). In addition, we denote by \(d^0_D=d_D(\emptyset)\) the cardinality of a smallest feasible set.



\begin{Th}
  For a ribbon graph~$R$, the number
  $d^0_{D(R)}+1$ coincides with the number of vertices of~$R$.
\end{Th}

\begin{proof}
  This statement follows from the fact that $D(R)*A$ is the delta-matroid of $R*A$ where $R*A$ is
  the partial duality of $R$ by the set $A$, see~\cite{generalized duality}. In order to obtain a ribbon graph with one vertex, we need
  to take for~$A$ a set containing a spanning tree on the vertices of $R$. The number of edges in this tree is one less than $d^0_{D(R)}$.
\end{proof}

The \emph{number $\bc(D)$ of connected components of the boundary of a delta-matroid $D=(E,  S)$} is the minimal $n \in
\mathbb N$ possessing the property that there exists a set $A\subset E$ of cardinality $n-1$ such that
$D*(E \setminus A)$ is a graphic delta-matroid.

\begin{Rem}
  It is easy to see that $d^0_{D}+1=\bc(D*E(D))$, where $E(D)$ is the ground set of $D$.
\end{Rem}

\begin{Cor}
  Let $R$ be a ribbon graph, then $\bc(R)=\bc(D(R))$.
\end{Cor}


Let~$D=(E,S)$ be a binary delta-matroid, and let $a,b\in E$.

The result of {\it sliding of the element~$a$ over the element~$b$\/}
is the set system $\widetilde D_{ab}=(E;\widetilde S_{ab})$, where
$\widetilde S_{ab}= S\Delta\{A\sqcup\{a\}| A\sqcup\{b\}\in S
\text{ and } A\subset E\sm\{a,b\}\}$.

This definition was given in \cite{moffatt2017handle}
and interpreted as the \emph{second Vassiliev move} in \cite{lando2016delta}.

The result of {\it exchanging the ends of the ribbons~$a,b$}
is the set system $D'_{ab}=(E; S'_{ab})$, where $ S'_{ab}=\widetilde{( S*b)}_{ab}*b$,
and this is the \emph{first Vassiliev move}.

\begin{Rem}[see \cite{lando2016delta} Proposition 4.5]
    The following statements about the Vassiliev moves are valid:
    \begin{itemize}
    \item the first Vassiliev move is an involution: \(({D'_{ab}})'_{ab}=D\);
    \item the second Vassiliev move is an involution: \(\widetilde{(\widetilde D_{ab})}_{ab}=D\);
    \item the first and the second Vassiliev moves commute: \((\widetilde D_{ab})'_{ab}=\widetilde{(D'_{ab})}_{ab}\)
    \end{itemize}
  \end{Rem}




\begin{Rem}\label{sec:basics-delta-matr-3}
  If 
  \(a\) is a coloop, then $\widetilde S_{ab}= S$, and $ S'_{ab}= S$.
\end{Rem}




\subsection{Transition polynomial for binary delta-matroids}
\label{sec:trans-polyn-binary}

In order to define the transition polynomial for binary delta-matriods,
we need two more operations.

Let \(D=(E,  S)\) be a \(\btu\)-matroid, and let \(u\in E\) be an element of its ground set. Then let us define the
{\em loop complementation} \(D+u\) of \(D\) on \(u\) by the formula
\(D+u=(E,  S\Delta\{F\cup u| F\in  S, u\not\in F\})\).

Below, operations on set systems are assumed to be applied from left to right,
so that, for example, \(M + u \setminus  u * v\) means \(((M + u) \setminus u) * v.\)

Define the {\em dual pivot} \(D\overline{*}u\) of a $\btu$-matroid~$D$ with respect
to an element~$u$ by
\(D\overline{*}u=D+u*u+u=D*u+u*u\).
Similarly, for a subset $A\subset E$ of the ground set, we set
\(D\overline{*}A=D+A*A+A=D*A+A*A\).



The following definition is a specialization of the definition
of weighted transition polynomial for delta-matroids in~\cite{BRHO}.

For a \(\Delta\)-matroid \(D\), we define its {\em transition polynomial} $Q(D)$ (with parameters \(s, t, -t\)) as

\[
  Q(D)=\sum_{E(D)=\Phi \sqcup X \sqcup \Psi} s^{|\Phi|} t^{|X|} (-t)^{|\Psi|}
  x^{d^0_{D+\Phi*X\overline{*}\Psi}},
\]
where summation is carried over all disjoint partitions of the ground
set~$E$ of~$D$ into three parts.

Our main result for delta-matroids is the following statement.

\begin{Th}\label{theorem-4-term-for-d-matroids}
  For an arbitrary binary \(\Delta\)-matroid \(D=(E, S)\) and arbitrary elements $a,b\in E$ in its ground set, we have
  \begin{equation}
    Q(D)-Q(D'_{ab})-Q( \widetilde D_{ab})+Q(\widetilde D'_{ab})=0\label{eq:2}
  \end{equation}
\end{Th}

The proof will require the following statement (Lemma~11 in~\cite{BRHO}).

\begin{Le}
  Let \(M=(E,S)\) be a set system, and let \(u\), \(v\in E\) such that \(u\not=v\).
  Then \(M + u \sm v =M \sm v + u\), \(M * u \sm v = M \sm v * u\), и \(M + u \sm u = M \sm u\).
\end{Le}

\begin{Prop}[see Lemma 2.11 in \cite{chun2019matroids}]
  \label{st-dD}
  Let \(D\) be a  \(\Delta\)-matroid and \(u\) an element of its ground set, then
  \(d^0_{D}=d^0_{D\sm u} - 1\) if \(u\) is a coloop, and \(d^0_{D}=d^0_{D\sm u}\), otherwise.
\end{Prop}

\begin{Le} 
  \label{le-bridge}
  Let \(D\) be a binary \(\Delta\)-matroid and
  suppose \(a,b,u\in E(D)\) are pairwise distinct  elements of its ground set.
  Then if \(u\) is a coloop
  for one of the \(\Delta\)-matroids
  in the set  \(\{D, D*a, D+a, D\overline{*}a, D'_{ab}, \widetilde D_{ab}, \widetilde D'_{ab}\}\), then
  it is a coloop for all of them.
\end{Le}
\begin{proof}
  It is easy to see that all the operations
  in the lemma are involutions, whence we have to prove   only sufficiency.

  It follows from the definitions of the operations that it suffices to prove the statement only for the operations
  \(D\mapsto D*a\), \(D\mapsto D+a\), and \(D\mapsto \widetilde{D}_{ab}\), 
  which is obvious by definition.
\end{proof}

The next proposition is an immediate corollary of Proposition~\ref{st-dD} and Lem\-ma~\ref{le-bridge}.
\begin{Prop}
  \label{prop:1}
  For an arbitrary binary delta-matroid \(D\) and pairwise distinct elements \(a,b,u\) in its ground set,
  the operation \(D\mapsto D\setminus u\) commutes with
  the first and the second Vassiliev moves on the elements
  $a,b$.
\end{Prop}

\begin{Le}\label{le-commutative}
  For an arbitrary binary delta-matroid \(D\) and pairwise distinct elements \(a,b,u\) in its ground set,
  the operations  \(D\mapsto D* u\),
  \(D\mapsto D+ u\), \(D\mapsto D\overline* u\)  commute with
  the first and the second Vassiliev moves on the elements
  $a,b$.
\end{Le}

\begin{proof}
  Since the second Vassiliev move is a composition of the
  first one and the operation~$*$, it suffices to check
  commutativity of the operations  \(D\mapsto D* u\)
  and  \(D\mapsto D+u\) with the first Vassiliev move.

  {Introduce the characteristic function \(\chi_{D}: 2^E\to\mathbb{Z}/2\mathbb{Z}\)
    of a \(\Delta\)-matroid \(D\), which takes a subset \(F\) of its base set to \(1\)
     if \(F\) is admissible and to \(0\) otherwise. Clearly, $D$ is uniquely determined by~$\chi_D$.

Now, \[
      \chi_{D'_{ab}}(F)=
      \begin{cases}
        \chi_{D}(F) \text{\hspace{79pt} if } \{a, b\}\not\subset F \\
        \chi_{D}(F) + \chi_{D}(F\setminus\{a, b\}) \quad\text{ if } \{a, b\}\subset F
      \end{cases}
    \]
    \[
      \chi_{D+u}(F)=\begin{cases}
        \chi_{D}(F) \text{\hspace{79pt} if } \{u\}\not\subset F \\
        \chi_{D}(F) + \chi_{D}(F\setminus \{u\}) \quad\text{ if } \{u\}\subset F
      \end{cases}
    \]

    It is easy to see that \(\chi_{(D'_{ab})+u}=\chi_{(D+u)'_{ab}}\) and that for any
    \(F\) such that \(F\not\ni a\), \(F\not\ni b\), \(F\not\ni u\) we have
    \begin{equation*}
      \begin{split}
    & \chi_{(D+u)'_{ab}}(F)=\chi_{D}(F) \\
    & \chi_{(D+u)'_{ab}}(F\cup\{a\})=\chi_{D}(F\cup\{a\}) \\
    & \chi_{(D+u)'_{ab}}(F\cup\{b\})=\chi_{D}(F\cup\{b\}) \\
    & \chi_{(D+u)'_{ab}}(F\cup\{u\})=\chi_{D}(F\cup\{u\})+\chi_{D}(F) \\
    & \chi_{(D+u)'_{ab}}(F\cup\{a, b\})=\chi_{D}(F\cup\{a, b\})+\chi_{D}(F) \\
    & \chi_{(D+u)'_{ab}}(F\cup\{a, u\})=\chi_{D}(F\cup\{a, u\})+\chi_{D}(F\cup\{a\}) \\
    & \chi_{(D+u)'_{ab}}(F\cup\{b, u\})=\chi_{D}(F\cup\{b, u\})+\chi_{D}(F\cup\{b\}) \\
    & \chi_{(D+u)'_{ab}}(F\cup\{a, b, u\})=\chi_{D}(F\cup\{a, b, u\})+\chi_{D}(F\cup\{u\})
      +\chi_{D}(F)
    \end{split}
  \end{equation*}
  (where summation on the right is taken in $\mathbb{Z}/2\mathbb{Z}$).}

    For the  pair \(D'_{ab}\) and \(*u\),
    \begin{eqnarray*}
      (D*u)'_{ab}&=&(E, [\{F\Delta u| F\in  S\}]\Delta\{F\cup \{a,b\}| F\in [\{F\Delta u| F\in  S\}],  \{ab\}\cap F = \emptyset\})\\
                 &=&(E, [\{F| F\Delta u\in  S\}]\Delta\{F\cup \{a,b\}| F\in [\{F| F\Delta u\in  S\}],  \{ab\}\cap F = \emptyset\})\\
                 &=&(E, \{F\Delta u |F\in [ S\Delta{F'\cup\{a,b\}| F'\in  S}]\})\\
                 &=&D'_{ab}*u.
    \end{eqnarray*}
\end{proof}

Now we can prove Theorem~\ref{theorem-4-term-for-d-matroids}
\begin{proof}
  %
  Let us pick a pair of distinct elements \(a, b\)  in the ground set \(E\) and define
  the polynomial \(Q_{\{a, b\}}(D)\) as
  \begin{equation}
    Q_{\{a, b\}}(D)=\sum_{\{a, b\}=E_{\phi} \sqcup E_{\chi} \sqcup E_{\psi}} s^{|E_{\phi}|} t^{|E_{\chi}|} (-t)^{|E_{\psi}|}
    x^{d^0_{D+E_{\phi}*E_\chi\overline{*}E_{\psi}}}.
    \label{eq:4}
  \end{equation}
  Now, we have
  \begin{eqnarray*}
    Q(D)&=&\sum_{E(D)\sm \{a, b\}=\Phi\sqcup X\sqcup \Psi}s^{|\Phi|}t^{|X|}(-t)^{|\Psi|}Q_{\{a, b\}}(D+\Phi*X\overline{*}\Psi),\\
    Q(D'_{ab})&=&\sum_{E(D)\sm \{a, b\}=\Phi\sqcup X\sqcup \Psi}s^{|\Phi|}t^{|X|}(-t)^{|\Psi|}Q_{\{a, b\}}(D'_{ab}+\Phi*X\overline{*}\Psi),\\
    Q(\widetilde{D}_{ab})&=&\sum_{E(D)\sm \{a, b\}=\Phi\sqcup X\sqcup \Psi}s^{|\Phi|}t^{|X|}(-t)^{|\Psi|}Q_{\{a, b\}}(\widetilde{D}_{ab}+\Phi*X\overline{*}\Psi),\\
    Q(\widetilde{D}'_{ab})&=&\sum_{E(D)\sm \{a, b\}=\Phi\sqcup X\sqcup \Psi}s^{|\Phi|}t^{|X|}(-t)^{|\Psi|}Q_{\{a, b\}}(\widetilde{D}'_{ab}+\Phi*X\overline{*}\Psi)
  \end{eqnarray*}

  Lemma~\ref{le-commutative} implies the following presentations, where the summations are carried over all partitions of the set $\{a,b\}$ into triples of disjoint subsets
  $\{a,b\}=E_\phi\sqcup E_\chi\sqcup E_\psi$:
  \begin{eqnarray*}
    Q_{\{a, b\}}(D+\Phi*X\overline{*}\Psi)&=&Q_{\{a, b\}}({D_1})\\
                                             &=&\sum_{\{a, b\}=E_{\phi} \sqcup E_{\chi} \sqcup E_{\psi}} s^{|E_{\phi}|} t^{|E_{\chi}|} (-t)^{|E_{\psi}|}
                                                 x^{d^0_{{D_1}+E_{\phi}*E_\chi\overline{*}E_{\psi}}},\\
    Q_{\{a, b\}}(D'_{ab}+\Phi*X\overline{*}\Psi)&=&Q_{\{a, b\}}({D_1}'_{ab})\\
                                             &=&\sum_{\{a, b\}=E_{\phi} \sqcup E_{\chi} \sqcup E_{\psi}} s^{|E_{\phi}|} t^{|E_{\chi}|} (-t)^{|E_{\psi}|}
                                                 x^{d^0_{{D_1'}_{ab}+E_{\phi}*E_\chi\overline{*}E_{\psi}}},\\
    Q_{\{a, b\}}(\widetilde{D}_{ab}+\Phi*X\overline{*}\Psi)&=&Q_{\{a, b\}}(\widetilde{D_1}_{ab})\\
                                             &=&\sum_{\{a, b\}=E_{\phi} \sqcup E_{\chi} \sqcup E_{\psi}} s^{|E_{\phi}|} t^{|E_{\chi}|} (-t)^{|E_{\psi}|}
                                                 x^{d^0_{\widetilde{D_1}_{ab}+E_{\phi}*E_\chi\overline{*}E_{\psi}}},\\
    Q_{\{a, b\}}(\widetilde{D}'_{ab}+\Phi*X\overline{*}\Psi)&=&Q_{\{a, b\}}(\widetilde{D_1}'_{ab})\\
                                             &=&\sum_{\{a, b\}=E_{\phi} \sqcup E_{\chi} \sqcup E_{\psi}} s^{|E_{\phi}|} t^{|E_{\chi}|} (-t)^{|E_{\psi}|}
                                                 x^{d^0_{\widetilde{D_1'}_{ab}+E_{\phi}*E_\chi\overline{*}E_{\psi}}},
  \end{eqnarray*}
  where \(D_1=D+\Phi*X\overline{*}\Psi\).

  It is therefore sufficient to show that for any partition of the set \(E(D)\setminus\{a, b\}\) into disjoint sets \(\Phi\), \(X\) and \(\Psi\),
  the equation
  \begin{equation}
    Q_{\{a, b\}}(D_1)-Q_{\{a, b\}}({D_1}'_{ab})-Q_{\{a, b\}}(\widetilde{D_1}_{ab})+Q_{\{a, b\}}(\widetilde{D_1}'_{ab})=0\label{eq:1}
  \end{equation}
  holds.  By Proposition~\ref{prop:1},
  the latter equation is equivalent (for an arbitrary \(u\not\in\{a, b\}\)) to the equation

  \[
    Q_{\{a, b\}}(D_1\sm u)-Q_{\{a, b\}}((D_1\sm u)'_{ab})-Q_{\{a, b\}}(\widetilde{D_1\sm u}_{ab})+Q_{\{a, b\}}(\widetilde{(D_1\sm u)}'_{ab})=0
  \]

  Therefore, we need to prove~Eq.(\ref{eq:1})
  only for delta-matroids with the ground set  \(\{a, b\}\).
  For any such delta-matriod \(D\), there exists a ribbon graph \(R\) such that \(D=D(R)\), and the
  assertion follows from the one for ribbon graphs.
\end{proof}

\end{document}